\documentclass[12pt]{article}

\usepackage{float}

\usepackage{pgfplots}
\usepackage{tikz}
\usetikzlibrary{calc,positioning,intersections,quotes,decorations.markings,arrows.meta, bending}
\usepackage{tkz-euclide}

\usepackage{graphicx}
\usepackage{comment}
\usepackage{mathrsfs}
\usepackage{nameref}
\usepackage[shortlabels,inline]{enumitem}
\usepackage{amsmath}
\usepackage{amssymb}
\usepackage{empheq}
\usepackage[shortlabels]{enumitem}
\setlist[enumerate]{nosep}
\usepackage[doc]{optional}
\usepackage{xcolor}
\usepackage[colorlinks=true,
            linkcolor=refkey,
            urlcolor=lblue,
            citecolor=red]{hyperref}
\usepackage{float}
\usepackage{soul}
\usepackage{graphicx}
\definecolor{labelkey}{rgb}{0,0.08,0.45}
\definecolor{refkey}{rgb}{0,0.6,0.0}
\definecolor{Brown}{rgb}{0.45,0.0,0.05}
\definecolor{lime}{rgb}{0.00,0.8,0.0}
\definecolor{lblue}{rgb}{0.5,0.5,0.99}

 \usepackage{mathpazo}

\usepackage{subcaption}
\colorlet{hlcyan}{cyan!30}

\usepackage{stmaryrd}
\usepackage{amssymb}
\makeatletter
\def\namedlabel#1#2{\begingroup
   \def\@currentlabel{#2}%
   \label{#1}\endgroup
}
\makeatother

\newcommand{\sepp}{\setlength{\itemsep}{-2pt}}

\newcommand{\weakly}{\ensuremath{\:{\rightharpoonup}\:}}

\newcommand{\nnn}{\ensuremath{{n\in{\mathbb N}}}}
\newcommand{\thalb}{\ensuremath{\tfrac{1}{2}}}
\newcommand{\halb}{\ensuremath{\frac{1}{2}}}
\newcommand{\menge}[2]{\big\{{#1}~\big |~{#2}\big\}}

\newcommand{\Menge}[2]{\left\{{#1}~\Big|~{#2}\right\}}

\newcommand{\fenv}[1]%
{\ensuremath{\,\overrightarrow{\operatorname{env}}_{#1}}}
\newcommand{\benv}[1]%
{\ensuremath{\,\overleftarrow{\operatorname{env}}_{#1}}}

\newcommand{\scal}[2]{\left\langle{#1},{#2}  \right\rangle}

\newcommand{\RR}{\ensuremath{\mathbb R}}
\newcommand{\myS}{\ensuremath{\mathbb S}}

\newcommand{\NN}{\ensuremath{\mathbb N}}

\newcommand{\conv}{\ensuremath{\operatorname{conv}\,}}

\newcommand{\cspan}{\ensuremath{\overline{\operatorname{span}}\,}}

\newcommand{\Id}{\ensuremath{\operatorname{Id}}}

\DeclarePairedDelimiterX\set[2]{ \{ }{ \}_{#2} }{#1}

\DeclarePairedDelimiterX\rb[1]{ ( }{ ) }{#1}

{\begin{list}{}{%
\settowidth{\labelwidth}{\textrm{#1~}}%
\setlength{\leftmargin}{\labelwidth+\labelsep}}}
{\end{list}}
\usepackage{amsthm}
\usepackage[capitalize,nameinlink]{cleveref}
\crefname{equation}{}{equations}
\crefname{figure}{Figure}{Figures}
\crefname{chapter}{Appendix}{chapters}
\crefname{item}{}{items}
\crefname{enumi}{}{}

\theoremstyle{definition}
\newtheorem{theorem}{Theorem}[section]
\newtheorem{lemma}[theorem]{Lemma}

\newtheorem{corollary}[theorem]{Corollary}

\newtheorem{proposition}[theorem]{Proposition}

\newtheorem{example}[theorem]{Example}
\newtheorem{ex}[theorem]{Example}
\newtheorem{fact}[theorem]{Fact}
\newtheorem{remark}[theorem]{Remark}

\usepackage{multirow}

\providecommand{\RR}{\mathbb{R}}

\providecommand{\conv}{\operatorname{conv}}

\providecommand{\Id}{\operatorname{{ Id}}}

\providecommand{\NN}{\mathbb{N}}

\providecommand{\Id}{\operatorname{Id}}

\providecommand{\spn}{\operatorname{span}}

\providecommand{\RR}{\mathbb{R}}

\providecommand{\NN}{\mathbb{N}}

\definecolor{myblue}{rgb}{.8, .8, 1}
  \newcommand*\mybluebox[1]{%
    \colorbox{myblue}{\hspace{1em}#1\hspace{1em}}}

\allowdisplaybreaks 

\begin{document}

\title{\textsc{
The projection onto the cross 
}}

\author{
Heinz H.\ Bauschke\thanks{
Mathematics, University
of British Columbia,
Kelowna, B.C.\ V1V~1V7, Canada. E-mail:
\texttt{heinz.bauschke@ubc.ca}.},~
Manish\ Krishan Lal\thanks{
Mathematics, University
of British Columbia,
Kelowna, B.C.\ V1V~1V7, Canada. E-mail:
\texttt{manish.krishanlal@ubc.ca}.},
and
Xianfu Wang\thanks{
Mathematics, University
of British Columbia,
Kelowna, B.C.\ V1V~1V7, Canada. E-mail:
\texttt{shawn.wang@ubc.ca}.}
}

\date{January 27, 2022} 
\maketitle

\vskip 8mm

\begin{abstract} \noindent
We consider the set of pairs of orthogonal vectors in Hilbert space, which 
is also called the cross because it is the union of the horizontal 
and vertical axes in the Euclidean plane when the underlying space is the real line. 
Crosses, which are nonconvex sets, play a significant role in various branches of nonsmooth analysis 
such as feasibility problems and optimization problems. 

In this work, we study crosses and show that in 
infinite-dimensional settings, they are never weakly (sequentially) closed.
Nonetheless, crosses do turn out to be proximinal (i.e., they always admit projections) and 
we provide explicit formulas for the projection onto the cross in all cases. 
\end{abstract}

{\small

}

\section{Introduction}
Throughout this paper, we assume that 
\begin{empheq}[box=\mybluebox]{equation}
    \text{$X$ is
    a real Hilbert space with inner product
    $\scal{\cdot}{\cdot}\colon X\times X\to\RR$, }
\end{empheq}
and induced norm $\|\cdot\|$.  
Consider the set $C$ defined by 
\begin{empheq}[box=\mybluebox]{equation}
\label{e:defC}
C := \menge{(x,y)\in X\times
X}{\scal{x}{y}=0},
\end{empheq}
where $X\times X$ denotes the product Hilbert space with 
inner product $\scal{(x,y)}{(u,v)} = \scal{x}{u}+\scal{y}{v}$. 
The set $C$ plays a role in optimization, although it does not seem to have 
a name that is universally used. 
We follow Kruger, Luke, and Thao's convention and refer to $C$ as \emph{cross} 
(see \cite[Example~2(a) on page~291]{KLT}). 
Here is a selection of situations in which crosses are used. 
\begin{itemize}
\item 
When $X=\RR$, then 
\begin{equation}
  C = \menge{(x,y)\in\RR^2}{xy=0}
\end{equation}
is of interest in the study of set regularity and 
feasibility problems; see, e.g., 
Kruger et al.'s \cite[Example~2(a) on page~291]{KLT}. 
\item 
When $X=\RR$, then 
$C$ is also known as the \emph{switching cone} in the study 
of mathematical programs by Liang and Ye; see, e.g., 
\cite[equation~(20) in Section~4]{LY}. 
\item 
When $X=\RR$ and one considers the \emph{(rectangular) hyperbola} defined by 
\begin{equation}
  xy=\alpha,
\end{equation}
where $\alpha>0$, then $C$ arises as the asymptotic case when $\alpha\to 0^+$. 
\item 
The set $C$ is a special case of the more general set 
$C_\gamma := \menge{(x,y)\in X\times X}{\scal{x}{y}=\gamma}$, where 
$X$ is finite-dimensional and $\gamma\in\RR$ is fixed: indeed, $C=C_0$. 
Sets of the form $C_\gamma$ are considered in nonnegative matrix factorization and also 
in deep learning --- see \cite[Sections~4.1--4.2 and Appendix~B]{Elser}. 
In that paper, Elser discusses also certain (but not all) cases of projecting onto $C_\gamma$ 
and computing projections numerically. He refers to $C_\gamma$ as a \emph{bilinear constraint} set. 
\item The study of conic optimization problems may lead to the set $C$.
If $K$ is a nonempty closed cone in $X$ and $K^\oplus$ is its dual cone, 
then the set 
\begin{equation}
  C \cap (K\times K^\oplus)
\end{equation}
is called the \emph{conic complementarity set}; 
see, e.g., Busseti, Moursi, and Boyd's \cite[Section~2]{BMB}.
\end{itemize}

In \cite{Elser}, the author is particularly interested in the 
\emph{projection} (nearest point mapping) associated with the cross in the context
of algorithms. 

Classically, the most famous condition 
\emph{sufficient for uniqueness of the projection} for a closed set
is \emph{convexity}.
Unfortunately, the cross $C$ is far from being convex: 

\begin{lemma}[\bf convex hull of $C$]
We have $\conv C = X\times X$.
\end{lemma}
\begin{proof}
Clearly, $X\times\{0\}\subseteq C$ and 
$\{0\}\times X\subseteq C$.
Let $(x,y)\in X\times X$.
Then $(2x,0)\in C$ and $(0,2y)\in C$;
hence, $(x,y)=\thalb(2x,0)+\thalb(0,2y)\in\conv C$
and we are done.
\end{proof}

Because $C$ is not convex, the question arises whether 
$C$ at least admits projections everywhere, i.e., 
whether $C$ is \emph{proximinal}. 

An often employed condition \emph{sufficient} for the existence of projections is 
weak closedness of the set (see, e.g., \cite[Proposition~3.14]{BC2017}). 
When $X$ is finite-dimensional, weak closedness \emph{characterizes} proximinality and 
is of course also equivalent to ordinary closedness
(see, e.g., \cite[Corollary~3.15]{BC2017}).
Unfortunately, when $X$ is infinite-dimensional, 
then $C$ is far from being weakly closed as the next result illustrates: 

\begin{lemma}[\bf weak (sequential) closure of $C$]
The set $C$ is closed. 
If $X$ is finite-dimensional, then $C$ is proximinal. 
If $X$ is infinite-dimensional, 
then $C$ is not weakly (sequentially) closed;
in fact, the weak (sequential) closure of $C$ is equal to $X\times X$. 
\end{lemma}
\begin{proof}
The continuity of the inner product immediately yields the closedness of $C$. 
So we assume that $X$ is infinite-dimensional and we 
let $(u_i)_{i\in I}$ be an orthonormal family in $X$ 
such that 
$\cspan \{u_i\}_{i\in I}= X$.
Next, let $J$ be a countably infinite subset of $I$, say
$J = \{i_n\}_{\nnn}$ for some sequence 
$(i_n)_\nnn$ in $I$ with pairwise distinct terms.
Set $(\forall \nnn)$ $e_n := u_{i_n}$. 
Then $e_n\weakly 0$ 
 by \cite[Example~2.32]{BC2017}. 
Set $S_J := \spn\{u_j\}_{j\in J} = \spn\{e_n\}_\nnn$, 
and let $(x,y)\in S_J\times S_J$. 
Now set 
$\zeta := \scal{x}{y}$,  
and 
$(\forall\nnn)$
$(x_n,y_n) := (x+\zeta e_n,y- e_n)$.
Then there exists $N\in\NN$ such that for all $n\geq N$, 
we have $e_n \in \{x,y\}^\perp$ and thus 
\begin{equation}
\scal{x_n}{y_n}
=\scal{x+\zeta e_n}{y-e_n}
=\scal{x}{y}-\zeta\scal{e_n}{e_n}=0. 
\end{equation}
Consequently, the sequence $(x_n,y_n)_{n\geq N}$ lies in $C$. 
On the other hand, $(x_n,y_n)\weakly (x,y)$. 
Altogether, 
$S_J\times S_J$ lies in the weak sequential closure of $C$.
Because $S_J\times S_J$ is convex, 
it follows from \cite[Theorem~3.34]{BC2017}
that its weak sequential closure
$\overline{S_J} \times \overline{S_J}$
is also contained in the weak sequential closure of $C$.
This is true for every countably infinite subset $J$ of $I$.
Finally, let $(x,y)\in X\times X$. 
Then there exists a countably infinite subset $J$ of $I$
such that $(x,y)\in \overline{S_J} \times \overline{S_J}$. 
Therefore,
the weak (sequential) closure of $C$ is equal to $X\times X$.
In particular, $C$ is neither weakly sequentially closed nor weakly closed. 
\end{proof}

After these explanations, we are now ready to state the main contribution of this paper: 
\emph{We will show that $C$ is in fact proximinal and we will also provide an 
explicit formula for the projection onto $C$.} 

The remainder of the paper is organized as follows. 
In \cref{sec:known}, we collect a few results that are known but will help in subsequent 
sections. Various auxiliary results are obtained in \cref{sec:aux} to make 
the proof of the main result painless. 
The main result (\cref{t:main}) is proved in \cref{sec:main}. 
The notation we employ is standard and follows
largely \cite{BC2017}.  
 
\section{Known results}
\label{sec:known}

In this section, we record some results which will make the 
proofs given in subsequent sections more clear.

\begin{fact}
\label{f:Bert1}
Let $f\colon X\to\RR$ and $h\colon X\to\RR$ be continuously Fr\'echet differentiable.
Consider the problem that asks to 
\begin{equation} 
\label{e:210805a}
\text{minimize}\;\; f(x)\;\; \text{subject to}\;\; h(x)=0.
\end{equation}
If $x^*\in X$ is a local minimizer of \cref{e:210805a} and 
$\nabla h(x^*)\neq 0$, 
then there exists a unique $\lambda^*\in \RR$ such that 
\begin{equation}
\label{necessary}
\nabla f(x^*) + \lambda \nabla h(x^*) = 0.
\end{equation}
\end{fact}
\begin{proof}
When $X$ is finite-dimensional, then this follows from 
\cite[Proposition~4.1.1]{Bertsekas}.
If $X$ is infinite-dimensional, then use \cite[Theorem~9.3.1 on page~243]{Luenberger}. 
\end{proof}

\begin{lemma}
\label{l:2x2}
If $\lambda\in\RR\smallsetminus\{-1,1\}$, then 
\begin{equation}
\begin{pmatrix}
\Id & \lambda\Id \\
\lambda\Id & \Id
\end{pmatrix}^{-1} = 
\frac{1}{1-\lambda^2}
\begin{pmatrix}
\Id & -\lambda\Id \\
-\lambda\Id & \Id
\end{pmatrix}, 
\end{equation}
where the block matrices are interpreted as linear operators on $X\times X$.
\end{lemma}
\begin{proof}
The result follows by a direct verification.
\end{proof}

We denote by 
\begin{equation}
\myS:=\menge{x\in X}{\|x\|=1}
\end{equation}
the \emph{unit sphere} in $X$. 
The next result provides a parametrization of the unit sphere in $\RR^{n}$. 
(See also \cite{Blumenson}.)

\begin{fact}[\bf parametrization of the sphere]
  \label{f:para}
Suppose that $X=\RR^n$, where $n\geq 2$, and let $\rho>0$. 
Then every point
\begin{equation}
x=(x_1,x_2,\ldots,x_n)\in\myS \subseteq\RR^n
\end{equation}
is uniquely described by its \emph{spherical coordinates}
$\theta_1,\ldots,\theta_{n-2}$ in $\left[0,\pi\right]$ 
and $\theta_{n-1}\in\left[0,2\pi\right[$ via\footnote{Recall the empty product convention 
which sets such products equal to $1$.}
\begin{equation}
(\forall i\in\{1,2,\ldots,n\})\quad 
x_i = \begin{cases}
\rho \cos(\theta_i)\prod_{j=1}^{i-1}\sin(\theta_j), 
&\text{if $i\leq n-2$;}\\
\rho \cos(\theta_{n-1})\prod_{j=1}^{n-2}\sin(\theta_j), 
&\text{if $i= n-1$;}\\
\rho \sin(\theta_{n-1})\prod_{j=1}^{n-2}\sin(\theta_j), 
&\text{if $i= n$.}
\end{cases}
\end{equation}
\end{fact}

\section{Auxiliary results}
\label{sec:aux}

This section lays out the preparatory work for our main result. 

We consider a point $(x_0,y_0)\in X\times X$. 
Recall our aim, which is to compute $P_C(x_0,y_0)$.
The following result will be useful later. 

\begin{lemma}
\label{l:Calt}
We have 
\begin{equation}
\label{e:201102a}
C = \bigcup_{\text{$U$ is a closed linear subspace of $X$}} U \times U^\perp;
\end{equation}
moreover, 
\begin{equation}
\label{e:201102b}
C = \bigcup_{\text{$U$ is a linear subspace of $X$  with $\dim U\leq 1$}} U \times U^\perp.
\end{equation}
Consequently, if $(x_0,y_0)\in X\times X$, then 
\begin{equation}
\label{e:reduc}
P_C(x_0,y_0) \subseteq 
\bigcup_{\text{$U$ is a closed linear subspace of $X$}} \big\{(P_Ux_0,P_{U^\perp}y_0)\big\}. 
\end{equation}
\end{lemma}
\begin{proof}
Let $(x,y) \in X\times X$.
If $(x,y)\in C$, i.e., $x \perp y$, 
 and we set $U=\RR x$, then 
$(x,y)\in U\times U^\perp$ and 
$U$ is a (closed) linear 
subspace of $X$ with $\dim U\leq 1$. 
Conversely, if $U$ is a closed linear subspace 
of $X$ and $(x,y)\in U\times U^\perp$, then 
$x\perp y$ and so $(x,y)\in C$. 
Altogether, we have verified 
\cref{e:201102a} and \cref{e:201102b}. 
The ``Consequently'' part now follows from \cref{e:201102a}.
\end{proof}

The following example illustrates that the inclusion 
\cref{e:reduc} may be strict.

\begin{example}
Suppose that $X=\RR$, and set $(x_0,y_0) := (1,0)$ 
and $U:=\{0\}$. 
Then $(x_0,y_0)\in C$ and thus 
$P_C(x_0,y_0)=\{(x_0,y_0)\} = \{(1,0)\}$. 
However, $U^\perp=\RR$ and 
$(P_Ux_0,P_{U^\perp}y_0) = (0,0)\notin P_C(x_0,y_0)$.  
\end{example}

We also note that 
when $(x_0,y_0)\in C$, then\footnote{
We should technically write $P_C(x_0,y_0)=\{(x_0,y_0)\}$; 
however, for convenience and readability, we will 
identify singleton sets with the vectors they contain.
}
$P_C(x_0,y_0)=(x_0,y_0)$.
Thus, for the remainder of this section, we focus on the case when $(x_0,y_0)\notin C$, i.e., 
\begin{empheq}[box=\mybluebox]{equation}
\label{e:210805b}
\scal{x_0}{y_0}\neq 0; 
\;\;\text{consequently,}\;\;
x_0\neq 0 \;\;\text{and}\;\; y_0\neq 0.
\end{empheq}
To determine $P_C(x_0,y_0)$, 
we introduce the objective and constraint functions 
\begin{empheq}[box=\mybluebox]{equation}
\label{eq:obj}
f(x,y) := \thalb\|x-x_0\|^2 + \thalb\|y-y_0\|^2
\quad\text{and}\quad
h(x,y) := \scal{x}{y},
\end{empheq}
which are obviously continuously Fr\'echet differentiable on $X\times X$. 
Indeed, for every $(x,y)\in X\times X$, we have 
\begin{equation}
\label{e:gradients}
\nabla f(x,y) = (x-x_0,y-y_0)
\quad\text{and}\quad
\nabla h(x,y) = (y,x). 
\end{equation}
Moreover, the points in $P_C(x_0,y_0)$ are precisely the 
solutions to the following optimization problem:
\begin{empheq}[box=\mybluebox]{equation}
\label{problem}
\text{minimize} \quad f(x,y) 
\quad\text{subject to} \quad h(x,y) = 0. 
\end{empheq}
Indeed, $C=\menge{(x,y)\in X\times X}{h(x,y)=0}$
while the optimal value of \cref{problem} is 
$\thalb d_C^2(x_0,y_0)=\inf_{(x,y)\in C}\thalb\|(x,y)-(x_0,y_0)\|^2$.

\begin{proposition}
\label{p:201101a}
$(0,0)$ is never a solution to \cref{problem}.
Consequently, $P_C^{-1}(0,0)=\{(0,0)\}$. 
\end{proposition}
\begin{proof}
Suppose to the contrary that $(0,0)$ solves \cref{problem}.
Then the optimal value of \cref{problem} is 
$\thalb\|x_0\|^2 + \thalb\|y_0\|^2$
and both terms in this sum are positive (by \cref{e:210805b}). 
Note that $(x_0,0)\in C$. 
But 
$f(x_0,0) = \thalb\|x_0\|^2 < 
\thalb\|x_0\|^2 + \thalb\|y_0\|^2 = f(0,0)$
because $y_0\neq 0$. 
But this contradicts the minimality of $f(0,0)$. 
Hence $(0,0)$ cannot be optimal. 
The ``Consequently'' part follows. 
\end{proof}

\begin{corollary}
\label{c:201101a}
Suppose that $(x,y)\in X\times X$ solves \cref{problem}. 
Then $\nabla h(x,y) \neq (0,0)$. 
\end{corollary}
\begin{proof}
By hypothesis, 
$(x,y)\in P_C(x_0,y_0)$. 
Suppose 
to the contrary that $\nabla h(x,y)=(0,0)$. 
Then, by \cref{e:gradients}, $(x,y)=(0,0)$. 
Now \cref{p:201101a} shows that $(x,y)$ cannot be a solution to 
\cref{problem} which is absurd. 
\end{proof}

\begin{theorem}
\label{t:201101a}
Suppose that $(x,y)\in X\times X$ solves \cref{problem}.
Then there exists a unique $\lambda\in\RR$ such that 
\begin{equation}
\label{e:210805c}
x+\lambda y = x_0
\quad\text{and}\quad
y+\lambda x = y_0.
\end{equation}
\end{theorem}
\begin{proof}
By \cref{f:Bert1}, there exists a unique $\lambda\in\RR$ such that 
$\nabla f(x,y) + \lambda \nabla h(x,y) = (0,0)$. 
Using \cref{e:gradients}, this turns into
$(x-x_0,y-y_0) + \lambda (y,x) = (0,0)$, 
i.e., \cref{e:210805c}. 
\end{proof}

\begin{proposition}
\label{p:201101b}
Suppose that $(x,y)\in X\times X$ and $\lambda\in\RR$ satisfy
\begin{equation}
\label{e:210805d}
x+\lambda y = x_0 
\quad\text{and}\quad 
y + \lambda x = y_0. 
\end{equation}
Then the following hold:
\begin{enumerate}
\item 
\label{p:201101b1}
If $x_0\neq y_0$, then $\lambda\neq 1$. 
\item 
\label{p:201101b2}
If $x_0\neq -y_0$, then $\lambda\neq -1$. 
\end{enumerate}
\end{proposition}
\begin{proof}
\cref{p:201101b1}:
We prove the contrapositive and thus assume that 
$\lambda =1$.
Then, by \cref{e:210805d}, 
$x_0 = x+(1)y = x+y = y+x = y+(1)x = y_0$. 
\cref{p:201101b2}:
Argue similarly to the proof of \cref{p:201101b1}.
\end{proof}

\begin{proposition}
\label{p:201101c}
Suppose that $(x,y)\in X\times X$  and $\lambda\in\RR\smallsetminus\{-1,1\}$
satisfy 
\begin{equation}
\label{e:201101a}
x+\lambda y = x_0 
\quad\text{and}\quad 
y + \lambda x = y_0. 
\end{equation}
Then 
\begin{equation}
x = \frac{1}{1-\lambda^2}(x_0-\lambda y_0)
\quad\text{and}\quad 
y = \frac{1}{1-\lambda^2}(y_0-\lambda x_0). 
\end{equation}
\end{proposition}
\begin{proof}
Combine \cref{e:201101a} with \cref{l:2x2}.
\end{proof}

\begin{corollary}
\label{c:201101e}
Suppose that $x_0\neq \pm y_0$ and $(x,y)\in X\times X$ solves \cref{problem}. 
Then there exists a unique $\lambda\in\RR\smallsetminus\{-1,1\}$ 
such that $x+\lambda y = x_0$, $y+\lambda x=y_0$, and 
\begin{equation}
  x = \frac{1}{1-\lambda^2}(x_0-\lambda y_0)
  \quad\text{and}\quad 
  y = \frac{1}{1-\lambda^2}(y_0-\lambda x_0). 
\end{equation}
\end{corollary}
\begin{proof}
The existence and uniqueness of $\lambda\in\RR$ such that 
$x+\lambda y = x_0$ and $y+\lambda x=y_0$
follows from \cref{t:201101a}. 
Next, \cref{p:201101b} implies that 
$\lambda\neq\pm 1$. 
Finally, apply \cref{p:201101c}. 
\end{proof}

\begin{proposition}
\label{p:201101d}
Suppose $\lambda\in\RR\smallsetminus\{\pm 1\}$, 
and set 
\begin{equation}
\label{e:p:201101d}
  x := \frac{1}{1-\lambda^2}(x_0-\lambda y_0)
  \quad\text{and}\quad 
  y := \frac{1}{1-\lambda^2}(y_0-\lambda x_0). 
\end{equation}
Then the following hold for $(x,y)$ defined in \cref{e:p:201101d}:
\begin{enumerate}
\item 
\label{p:201101d0}
$x+\lambda y = x_0$ and $y+\lambda x = y_0$. 
\item 
\label{p:201101d0.5}
The objective function $f$ defined in \cref{eq:obj} evaluated at $(x,y)$ is 
\begin{equation}
f(x,y) = \frac{\lambda^2}{2(1-\lambda^2)^2}
\big((1+\lambda^2)(\|x_0\|^2+\|y_0\|^2)-4\lambda\scal{x_0}{y_0} \big).
\end{equation}
\item 
\label{p:201101d1}
$\scal{x}{y}=0$
if and only if 
$(1+\lambda^2)\scal{x_0}{y_0}
= \lambda(\|x_0\|^2 + \|y_0\|^2)$.
\item 
\label{p:201101d1.5}
If $\scal{x}{y}=0$, then 
\begin{equation}
f(x,y) = \thalb\lambda\scal{x_0}{y_0}.
\end{equation}
\end{enumerate}
\end{proposition}

\begin{proof}
\cref{p:201101d0}: 
This is an easy algebraic verification. 

\cref{p:201101d0.5}: 
For convenience, set 
\begin{equation}
\label{e:deftau}
  \tau := 
  \|x_0-\lambda y_0\|^2
  +
  \|y_0-\lambda x_0\|^2 
  = 
  (1-\lambda^2)^2\big(\|x\|^2 + \|y\|^2\big). 
\end{equation}
Then 
\begin{subequations}
\label{e:taualt}
\begin{align}
\tau &= 
\|x_0-\lambda y_0\|^2 + 
\|y_0-\lambda x_0\|^2 \\
&= 
\|x_0\|^2 -2\lambda\scal{x_0}{y_0} + \lambda^2\|y_0\|^2
+ \|y_0\|^2 -2\lambda\scal{y_0}{x_0} + \lambda^2\|x_0\|^2\\
&= (1+\lambda^2)(\|x_0\|^2+\|y_0\|^2)-4\lambda\scal{x_0}{y_0}.
\end{align}
\end{subequations}
Using 
\cref{eq:obj}, 
\cref{e:p:201101d}, 
\cref{e:deftau},
and 
\cref{e:taualt},
we obtain 
\begin{subequations}
\begin{align}
f(x,y)
&= \thalb \|x-x_0\|^2 + \thalb \|y-y_0\|^2\\
&= \frac{1}{2}\left\| \frac{x_0-\lambda y_0}{1-\lambda^2}-x_0\right\|^2
+ \frac{1}{2}\left\| \frac{y_0-\lambda x_0}{1-\lambda^2}-y_0\right\|^2 \\
&= \frac{1}{2(1-\lambda^2)^2}
\big(\|x_0-\lambda y_0 - (1-\lambda^2)x_0\|^2 + 
\|y_0-\lambda x_0-(1-\lambda^2)y_0\|^2 \big)\\
&= \frac{\lambda^2}{2(1-\lambda^2)^2}
\big(\|\lambda x_0- y_0\|^2 + 
\|\lambda y_0- x_0\|^2 \big)\\
&= \frac{\lambda^2}{2(1-\lambda^2)^2}
\tau \\
&= \frac{\lambda^2}{2(1-\lambda^2)^2}
\big((1+\lambda^2)(\|x_0\|^2+\|y_0\|^2)-4\lambda\scal{x_0}{y_0} \big).
\end{align}
\end{subequations}

\cref{p:201101d1}: 
Because $1-\lambda^2 \neq 0$ 
and using \cref{e:p:201101d}, 
we have the following equivalences:
\begin{subequations}
\begin{align}
\scal{x}{y} = 0 
&\Leftrightarrow 
\scal{x_0-\lambda y_0}{y_0-\lambda x_0} = 0 \\
&\Leftrightarrow 
\scal{x_0}{y_0} - \lambda\|x_0\|^2 - \lambda\|y_0\|^2 + \lambda^2\scal{x_0}{y_0}=0\\
&\Leftrightarrow 
(1+\lambda^2)\scal{x_0}{y_0}
= \lambda(\|x_0\|^2 + \|y_0\|^2).
\end{align}
\end{subequations}

\cref{p:201101d1.5}: 
Suppose that $\scal{x}{y}=0$. 
By \cref{p:201101d1}, 
\begin{equation}
\label{e:201101b}
\lambda\scal{x_0}{y_0}
= 
\frac{\lambda^2}{1+\lambda^2}\big(\|x_0\|^2 + \|y_0\|^2\big).
\end{equation}
It thus follows from \cref{e:201101b} that 
\begin{subequations}
\label{e:210805e}
\begin{align}
&\hspace{-1cm}
\big(1+\lambda^2\big)\big(\|x_0\|^2+\|y_0\|^2\big) - 4\lambda\scal{x_0}{y_0}\\
&= 
\big(1+\lambda^2\big)\big(\|x_0\|^2+\|y_0\|^2\big) - 4\frac{\lambda^2}{1+\lambda^2}\big(\|x_0\|^2 + \|y_0\|^2\big)\\
&= 
\frac{\|x_0\|^2+\|y_0\|^2}{1+\lambda^2}\big((1+\lambda^2)^2-4\lambda^2 \big)\\
&= 
\frac{\|x_0\|^2+\|y_0\|^2}{1+\lambda^2}\big(1+\lambda^4+2\lambda^2-4\lambda^2 \big)\\
&= 
\frac{\|x_0\|^2+\|y_0\|^2}{1+\lambda^2}\big(1+\lambda^4-2\lambda^2\big)\\
&= 
\frac{\|x_0\|^2+\|y_0\|^2}{1+\lambda^2}\big(1-\lambda^2\big)^2. 
\end{align}
\end{subequations}
Using \cref{p:201101d0.5}, \cref{e:210805e}, and \cref{e:201101b}, we conclude that 
\begin{subequations}
\begin{align}
f(x,y) &= \frac{\lambda^2}{2(1-\lambda^2)^2}\frac{\|x_0\|^2+\|y_0\|^2}{1+\lambda^2}\big(1-\lambda^2\big)^2\\
&=\frac{\lambda^2}{1+\lambda^2}\frac{\|x_0\|^2+\|y_0\|^2}{2}\\
&=\thalb\lambda\scal{x_0}{y_0}.
\end{align}
\end{subequations}
The proof is complete.
\end{proof}

\begin{proposition}
\label{p:quad}
Consider the equation 
\begin{equation}
\label{e:quadlambda}
  \scal{x_0}{y_0}\lambda^2
- (\|x_0\|^2 + \|y_0\|^2)\lambda 
+ \scal{x_0}{y_0} = 0, 
\end{equation}
which is a quadratic polynomial with respect to the variable $\lambda$.
Then \cref{e:quadlambda}
has (possibly distinct) real roots 
\begin{subequations}
\label{e:quadsol}
\begin{align}
\lambda_\pm &:= \frac{\|x_0\|^2 + \|y_0\|^2 
\pm \sqrt{(\|x_0\|^2 + \|y_0\|^2)^2 - 4\scal{x_0}{y_0}^2}}{2\scal{x_0}{y_0}}\\
&= \frac{\|x_0\|^2 + \|y_0\|^2 \pm \|x_0+y_0\|\|x_0-y_0\|}{2\scal{x_0}{y_0}}
\end{align}
\end{subequations}
which satisfy 
$\lambda_+\lambda_-=1$. 
Moreover, 
$\lambda_+\neq\lambda_-$ $\Leftrightarrow$ $x_0\neq \pm y_0$, 
\begin{equation}
\lambda_+\scal{x_0}{y_0}\geq \lambda_-\scal{x_0}{y_0}>0, 
\end{equation}
and the left inequality is strict if and only if $x_0\neq \pm y_0$. 
\end{proposition}
\begin{proof}
First note that 
$\|x_0\mp y_0\|^2 \geq 0$
$\Leftrightarrow$
$\|x_0\|^2 + \|y_0\|^2 \geq \pm 2 \scal{x_0}{y_0}$
$\Leftrightarrow$
$\|x_0\|^2 + \|y_0\|^2 \geq 2 |\negthinspace\scal{x_0}{y_0}\negthinspace|$
which makes the discriminant of 
\cref{e:quadlambda} nonnegative, and equal to $0$ $\Leftrightarrow$ 
$x_0=\pm y_0$. 
Hence, the roots of \cref{e:quadlambda} are real, 
and distinct if and only if $x_0\neq \pm y_0$. 
Second, Vieta's formulas give $\lambda_+\lambda_- = \scal{x_0}{y_0}/\scal{x_0}{y_0}=1$,
so both $\lambda_+$ and $\lambda_-$ are nonzero and they have the same sign. 
Next, the quadratic formula yields 
\begin{subequations}
\begin{align}
\lambda_\pm &= 
\frac{\|x_0\|^2 + \|y_0\|^2 
\pm \sqrt{(\|x_0\|^2 + \|y_0\|^2)^2 - 4\scal{x_0}{y_0}^2}}{2\scal{x_0}{y_0}}\\
&= \frac{\|x_0\|^2 + \|y_0\|^2 \pm \sqrt{\big(\|x_0\|^2 + \|y_0\|^2+2\scal{x_0}{y_0}\big)\big(\|x_0\|^2 + \|y_0\|^2-2\scal{x_0}{y_0}\big)}}{2\scal{x_0}{y_0}}\\
&= \frac{\|x_0\|^2 + \|y_0\|^2 \pm \|x_0+y_0\|\|x_0-y_0\|}{2\scal{x_0}{y_0}}, 
\end{align}
\end{subequations}
which yields \cref{e:quadsol}. 
Clearly, $\lambda_+\scal{x_0}{y_0}>0$ and 
$\lambda_+\scal{x_0}{y_0}\geq\lambda_-\scal{x_0}{y_0}$. 
Finally, as observed above, $\lambda_+$ and $\lambda_-$ have the same sign; hence,  
$\lambda_-\scal{x_0}{y_0}$ is positive as well. 
\end{proof}

\section{Main result}

\label{sec:main}

In this section, we derive formulas for $P_C$ in all possible cases.
We start with the trivial case.

\begin{proposition}
\label{p:trivial}
Suppose that $(x_0,y_0)\in X\times X$ satisfies $\scal{x_0}{y_0}=0$.
Then 
\begin{equation}
P_C(x_0,y_0)=(x_0,y_0).
\end{equation}
\end{proposition}
\begin{proof}
This is clear because $(x_0,y_0)\in C$ be definition of $C$ (see \cref{e:defC}). 
\end{proof}

We now turn to a much more involved case the proof of which won't be too 
complicated because of our preparatory work in \cref{sec:aux}.

\begin{proposition}
\label{p:single}
Suppose that 
$(x_0,y_0)\in X\times X$ satisfies
$\scal{x_0}{y_0}\neq 0$ and $x_0\neq\pm y_0$.
Then 
\begin{equation}
\lambda := \frac{\|x_0\|^2 + \|y_0\|^2 
- \|x_0+y_0\|\|x_0-y_0\|}{2\scal{x_0}{y_0}} \neq\pm 1, 
\end{equation}
\begin{equation}
\label{e:yay1}
P_C(x_0,y_0) 
= \frac{1}{1-\lambda^2}\big(x_0-\lambda y_0,y_0-\lambda x_0\big)
\end{equation}
is a \emph{singleton}, 
and 
\begin{equation}
\label{e:yay2}
\thalb d_C^2(x_0,y_0) = 
\thalb\lambda\scal{x_0}{y_0} = 
\frac{\|x_0\|^2 + \|y_0\|^2 
- \|x_0+y_0\|\|x_0-y_0\|}{4}. 
\end{equation}
\end{proposition}
\begin{proof}
Assume that 
\begin{equation}
\label{e:210806a}
(x,y)\in P_C(x_0,y_0). 
\end{equation}
(We will show that $P_C(x_0,y_0)\neq\varnothing$ later in this proof.)
Then $(x,y)$ solves \cref{problem}.
>From \cref{c:201101e},
there exists a unique $\lambda\in\RR\smallsetminus\{\pm 1\}$ 
such that $x+\lambda y = x_0$, $y+\lambda x=y_0$, and 
\begin{equation}
\label{e:210806b}
  x = \frac{1}{1-\lambda^2}(x_0-\lambda y_0)
  \quad\text{and}\quad 
  y = \frac{1}{1-\lambda^2}(y_0-\lambda x_0). 
\end{equation}
Because $(x,y)\in C$, \cref{p:201101d}\cref{p:201101d1} shows that 
 of the quadratic equation
\begin{equation}
\label{e:210806c}
\text{$\lambda$ is a solution of}\;\;
(1+\mu^2)\scal{x_0}{y_0} = \mu\big(\|x_0\|^2+\|y_0\|^2\big), 
\end{equation}
which is a quadratic equation in the real variable $\mu$. 
Set
\begin{equation}
\label{e:210806d}
\lambda_\pm := 
 \frac{\|x_0\|^2 + \|y_0\|^2 \pm \|x_0+y_0\|\|x_0-y_0\|}{2\scal{x_0}{y_0}},
\end{equation}
which are the two (possibly distinct) roots of 
the quadratic equation in \cref{e:210806c}. 
By \cref{e:210806c} and \cref{p:quad}, 
we deduce that $\lambda\in\{\lambda_-,\lambda_+\}$ and 
that $\lambda_+\lambda_-=1$. 
Because $x_0\neq\pm y_0$, \cref{p:quad} also yields 
$\lambda_+\neq\lambda_-$ and 
\begin{equation}
\label{e:210806f}
\lambda_+\scal{x_0}{y_0}>\lambda_-\scal{x_0}{y_0}. 
\end{equation}
Hence neither $\lambda_+$ nor $\lambda_-$ is equal to 
$\pm 1$. 
Now we (well) define 
\begin{equation}
\label{e:210806e}
  x_\pm = \frac{1}{1-\lambda_\pm^2}(x_0-\lambda_\pm y_0)
  \quad\text{and}\quad 
  y_\pm = \frac{1}{1-\lambda_\pm^2}(y_0-\lambda_\pm x_0). 
\end{equation}
In view of \cref{e:210806b}, $x\in \{x_-,x_+\}$. 
Because $\lambda_\pm$ solve the quadratic equation in 
\cref{e:210806c}, we deduce from \cref{p:201101d}\ref{p:201101d1} that 
$\scal{x_+}{y_+} = \scal{x_-}{y_-}=0$, i.e., 
$(x_+,y_+)$ and $(x_-,y_-)$ both belong to $C$. 
Recalling the definition of $f$ from \cref{eq:obj}, 
we note that \cref{p:201101d}\cref{p:201101d1.5} and 
\cref{e:210806f} result in 
$f(x_+,y_+)>f(x_-,y_-)$.
On the other hand, 
we know that $(x,y)$ is either 
$(x_+,y_+)$ or 
$(x_-,y_-)$.
Altogether, because we've assumed at the beginning of the proof in \cref{e:210806a} 
that $(x,y)$ is a minimizer of $f$, we conclude that 
$(x,y)=(x_-,y_-)$. This yields \cref{e:yay1}.
Moreover, \cref{e:yay2} follows because 
\begin{equation}
\label{e:210806g}
f(x,y) = f(x_-,y_-)=\frac{\lambda_-\scal{x_0}{y_0}}{2} 
= \frac{\|x_0\|^2+\|y_0\|^2 - \|x_0+y_0\|\|x_0-y_0\|}{4}
\end{equation}
by \cref{e:210806d}. 

It remains now to show that $P_C(x_0,y_0)\neq \varnothing$.
Let $U$ be a closed linear subspace of $X$. 
In view of \cref{e:reduc}, \cref{e:210806g}, and \cref{eq:obj} it suffices to show that 
$f(x,y) \stackrel{?}{\leq} f(P_Ux_0,P_{U^\perp}y_0)$, i.e., 
\begin{equation}
\label{e:holy}
  \frac{\|x_0\|^2+\|y_0\|^2 - \|x_0+y_0\|\|x_0-y_0\|}{4}
  \stackrel{?}{\leq}
\halb\|P_Ux_0-x_0\|^2 + \halb\|P_{U^\perp}y_0-y_0\|^2.
\end{equation}
To this end, recall that the reflector $R_U := P_U-P_{U^\perp}$ 
is a linear isometry. 
Using Cauchy-Schwarz, we thus estimate
\begin{subequations}
\begin{align}
&\hspace{-1cm}\|x_0+y_0\|\|x_0-y_0\|\\
&=
\|x_0+y_0\|\|R_U(x_0-y_0)\|\\
&\geq 
\scal{x_0+y_0}{R_U(x_0-y_0)}\\
&=
\scal{x_0+y_0}{P_U(x_0-y_0)-P_{U^\perp}(x_0-y_0)}\\
&=
\scal{x_0+y_0}{P_U(x_0-y_0)}
+\scal{x_0+y_0}{P_{U^\perp}(y_0-x_0)}\\
&=
\scal{P_U(x_0+y_0)}{P_U(x_0-y_0)}
+\scal{P_{U^\perp}(x_0+y_0)}{P_{U^\perp}(y_0-x_0)}\\
&=
\scal{P_Ux_0+P_Uy_0}{P_Ux_0-P_Uy_0}
+\scal{P_{U^\perp}x_0+P_{U^\perp}y_0}{P_{U^\perp}y_0-P_{U^\perp}x_0}\\
&=\|P_Ux_0\|^2-\|P_{U}y_0\|^2
+\|P_{U^\perp}y_0\|^2-\|P_{U^\perp}x_0\|^2.
\end{align}
\end{subequations}
This implies
\begin{equation}
\|P_Ux_0\|^2 +\|P_{U^\perp}y_0\|^2 - \|x_0+y_0\|\|x_0-y_0\|
\leq \|P_{U}y_0\|^2 + \|P_{U^\perp}x_0\|^2.
\end{equation}
Adding to this $\|P_{U^\perp}x_0\|^2+\|P_Uy_0\|^2$ yields 
\begin{equation}
\label{e:210806h}
\|x_0\|^2 +\|y_0\|^2 - \|x_0+y_0\|\|x_0-y_0\|
\leq 2\|P_{U}y_0\|^2 + 2\|P_{U^\perp}x_0\|^2.
\end{equation}
Finally, dividing \cref{e:210806h} by $4$ gives 
\cref{e:holy} and we are done. 
\end{proof}

\begin{proposition}
\label{p:multi}
Suppose that $(x_0,y_0)\in X\times X$ satisfies 
$x_0=\pm y_0$. 
Then 
\begin{equation}
\label{e:210807b}
P_C(x_0,y_0) = \bigcup_{\text{$U$ is a closed subspace of $X$}}
\big\{(P_Ux_0,P_{U^\perp}y_0) \big\};
\end{equation}
this can also be written as 
\begin{equation}
\label{e:210807c}
P_C(x_0,y_0) = \big\{(0,y_0)\big\} \cup 
\Menge{(0,y_0) + \big(\scal{u}{x_0}u,-\scal{u}{y_0}u\big)}{u\in\myS},
\end{equation}
where $\myS = \menge{z\in X}{\|z\|=1}$ is the unit sphere of $X$. 
\end{proposition}
\begin{proof}
In view of \cref{l:Calt}, 
let $U$ be an arbitrary closed subspace of $X$. 
Then 
\begin{equation}
P_{U\times U^\perp}(x_0,y_0) = \big(P_Ux_0,P_{U^\perp}y_0\big). 
\end{equation}
Let $f$ be defined as in \cref{eq:obj}. 
Using the fact that $P_U+P_{U^\perp} = \Id$ in 
\cref{e:220105a}, 
the assumption that $x_0=\pm y_0$ in \cref{e:210807a}, 
 and the linearity of $P_U$ in 
\cref{e:220105b},  
we see that 
\begin{subequations}
\begin{align}
f\big(P_Ux_0,P_{U^\perp}y_0\big)
&=\thalb \|x_0-P_Ux_0\|^2 + \thalb\|y_0-P_{U^\perp}y_0\|^2\\
&=\thalb \|P_{U^\perp}x_0\|^2 + \thalb\|P_Uy_0\|^2\label{e:220105a}\\
&=\thalb \|P_{U^\perp}x_0\|^2 + \thalb\|P_U(\pm x_0)\|^2\label{e:210807a}\\
&=\thalb \|P_{U^\perp}x_0\|^2 + \thalb\|P_Ux_0\|^2\label{e:220105b}\\
&=\thalb\|x_0\|^2\\
&=\tfrac{1}{4}\|x_0\|^2 + \tfrac{1}{4}\|y_0\|^2
\end{align}
\end{subequations}
is \emph{independent} of $U$!
This proves \cref{e:210807b}.

We now tackle \cref{e:210807c} via \cref{e:201102b}. 
So let $U$ be a linear subspace of $X$ with $\dim U \leq 1$.

If $U=\{0\}$, then 
$U^\perp=X$ and $(P_Ux_0,P_{U^\perp} y_0) = (0,y_0)$
which is the first term on the right side of \cref{e:210807c}. 

Now assume that $\dim U=1$, say $U=\RR u$, where $u\in\myS$. 
Then $\|u\|=1$, 
$P_Ux_0=\scal{u}{x_0}u$
and $P_{U^\perp}y_0 = y_0 -\scal{u}{y_0}u$.
Hence
\begin{equation}
  \big(P_Ux_0,P_{U^\perp}y_0\big)
  = (0,y_0) + \big(\scal{u}{x_0}u,-\scal{u}{y_0}u\big)
\end{equation}
which yields 
the second term on the right side of \cref{e:210807c}. 
\end{proof}

Let us summarize our work in one convenient theorem.

\begin{theorem}[\bf main result]
\label{t:main}
The set $C$ defined in \cref{e:defC} is proximinal.
Let $(x_0,y_0)\in X\times X$.
Then \emph{exactly one} of the following three cases occurs.
\begin{enumerate}
\item 
\label{t:main1}
$\scal{x_0}{y_0}=0$, $\thalb d_C^2(x_0,y_0)=0$, and 
\begin{equation}
\label{e:210807e}
P_C(x_0,y_0)=(x_0,y_0). 
\end{equation}
\item 
\label{t:main2}
$\scal{x_0}{y_0}\neq 0$, $x_0\neq\pm y_0$, 
\begin{equation}
\label{e:210807d}
P_C(x_0,y_0) 
= \frac{1}{1-\lambda^2}\big(x_0-\lambda y_0,y_0-\lambda x_0\big), 
\end{equation}
and 
\begin{equation}
\thalb d_C^2(x_0,y_0) = 
\thalb\lambda\scal{x_0}{y_0} = 
\frac{\|x_0\|^2 + \|y_0\|^2 
- \|x_0+y_0\|\|x_0-y_0\|}{4}, 
\end{equation}
where 
\begin{equation}
\label{e:dalambda}
\lambda := \frac{\|x_0\|^2 + \|y_0\|^2 
- \|x_0+y_0\|\|x_0-y_0\|}{2\scal{x_0}{y_0}} \neq\pm 1. 
\end{equation}
\item 
\label{t:main3}
$\scal{x_0}{y_0}\neq 0$, $x_0=\pm y_0$, 
$\thalb d_C^2(x_0,y_0)=\frac{1}{4}(\|x_0\|^2+\|y_0\|^2)$, 
and 
\begin{equation}
P_C(x_0,y_0) = \big\{(0,y_0)\big\} \cup 
\Menge{(0,y_0) + \big(\scal{u}{x_0}u,-\scal{u}{y_0}u\big)}{u\in\myS} 
\end{equation}
is not a singleton, 
where $\myS = \menge{z\in X}{\|z\|=1}$ is the unit sphere of $X$. 
\end{enumerate}
\end{theorem}
\begin{proof}
Combine \cref{p:trivial}, \cref{p:single}, and \cref{p:multi}. 
\end{proof}

\begin{remark}
Several results regarding \cref{t:main} are in order.
\begin{enumerate}
  \item {\bf (relationship to Elser's work)} Case~\cref{t:main2} was considered by Elser who obtained the basic 
  structure of \cref{e:210807d} in a more general setting; see 
  \cite[equation~(19)]{Elser}. However, his analysis is carried out in 
  the finite-dimensional setting. And neither was the
  explicit formula for $\lambda$ in \cref{e:dalambda} presented nor the 
  case~\cref{t:main3} discussed. 
  \item A convenient \emph{selection} of $P_C$ in case~\cref{t:main3} is 
  $(0,y_0)$ or $(x_0,0)$. 
  \item If we work in $X=\RR^n$ and we require the complete projection in 
  case~\cref{t:main3}, then we may invoke \cref{f:para}. 
  \item It is possible to subsume case~\cref{t:main1} into case~\cref{t:main2} by 
  setting $\lambda=0$ and using \cref{e:210807d} to obtain \cref{e:210807e}. 
  \item A tight \emph{injective} parametrization in case~\cref{t:main3} is
\begin{equation}
 P_C(x_0,y_0) = \big\{(0,y_0)\big\} \uplus \biguplus_{u\in\myS\;\text{and}\;\scal{u}{x_0}>0}
 \Big\{(0,y_0) + \big(\scal{u}{x_0}u,-\scal{u}{y_0}u\big)\Big\}, 
 \end{equation}
where ``$\uplus$'' denotes disjoint union. 
Hence the \emph{cardinality} of $P_C(x_0,y_0)$ is the same as the cardinality 
of $\myS$. 
\end{enumerate}
\end{remark}

When $X=\RR$, then the reader may confirm that \cref{t:main} simplifies to the following:

\begin{example}
\label{ex:cross}
Suppose that $X=\RR$. Then 
\begin{equation}
C=\menge{(x,y)\in\RR^2}{xy=0}
= (\RR\times\{0\})\cup (\{0\}\times\RR). 
\end{equation}
Let $(x_0,y_0)\in\RR^2$. 
Then exactly one of the following cases holds.
\begin{enumerate}
\item $|x_0|\neq |y_0|$ and 
\begin{equation}
P_C(x_0,y_0) = 
\begin{cases}
(0,y_0), &\text{if $|x_0|<|y_0|$;}\\
(x_0,0), &\text{if $|x_0|>|y_0|$.}
\end{cases}
\end{equation}
\item $|x_0|=|y_0|$ and 
\begin{equation}
  P_C(x_0,y_0) = \big\{(x_0,0),(0,y_0)\big\}. 
\end{equation}
\end{enumerate}
For an illustration, see \cref{fig}. 
\end{example}

\begin{figure}[H]
  \centering
\begin{tikzpicture}
\begin{axis}
[width=6in,axis equal image,
  axis lines=middle,
  xmin=-2,xmax=2,samples=201,
  xlabel=$x$,ylabel=$y$,
  ymin=-2,ymax=2,
  restrict y to domain=-2:2,
  enlargelimits={abs=0.5cm},
  axis line style={latex-latex},
  ticklabel style={font=\tiny,fill=white},
  xtick={\empty},ytick={\empty},
  xlabel style={at={(ticklabel* cs:1)},anchor=north west},
  ylabel style={at={(ticklabel* cs:1)},anchor=south west},
]

\node[label=right:{}] at (axis cs: {2.5*cos(45)},{2.5*sin(45)})       (a){$y=x$};
\node[label=above:{}] at (axis cs: {2.5*cos(135)},{2.5*sin(135)}) (b){$y=-x$};
\node[label=below left:{$(0,0)$}] at (axis cs: 0,0)     (P){};
\node[] at (axis cs: {2.5*cos(-135)},{2.5*sin(-135)}) (e){}; 
\node[] at (axis cs: {2.5*cos(-45)},{2.5*sin(-45)})  (f){};

\node [fill,circle, scale=0.2, label=right:{$(x_{0},y_{0})$}] (a1) at (axis cs: {cos(45)},{sin(45)})  {$(x_{0},y_{0})$};
\node [fill,circle, scale=0.2, label=below:{$(x_{0},0)$}] (a2) at (axis cs: {cos(45)},{sin(0)})  {$(x_{0},0)$};
\node [fill,circle, scale=0.2, label=left:{$(0,y_{0})$}] (a3) at (axis cs: {cos(90)},{sin(45)})  {$(0,y_{0})$};
\node [fill,circle, scale=0.2] (a4) at (axis cs: {2.5*cos(45)},{1.5*sin(45)})  {$(x_{0},y_{0})$};
\node [fill,circle, scale=0.2] (a5) at (axis cs: {2.5*cos(45)},{2*sin(0)})  {$(x_{0},y_{0})$};
\node [fill,circle, scale=0.2] (a6) at (axis cs: {-1.1*cos(45)},{2.2*sin(45)})  {$(x_{0},y_{0})$};
\node [fill,circle, scale=0.2] (a7) at (axis cs: {cos(90)},{2.2*sin(45)})  {$(x_{0},y_{0})$};
\node [fill,circle, scale=0.2] (a8) at (axis cs: {-1.8*cos(45)},{sin(0)})  {$(x_{0},0)$};
\node[fill,circle, scale=0.2] at (axis cs: {1.8*cos(-135)},{1.3*sin(-135)}) (a9) {$(x_{0},y_{0})$};
\node[fill,circle, scale=0.2] at (axis cs: {2*cos(-45)},{2*sin(-45)}) (a10) {$(x_{0},y_{0})$};
\node[fill,circle, scale=0.2] at (axis cs: {2*cos(-45)},{0*sin(-135)}) (a11) {$(x_{0},y_{0})$};
\node[fill,circle, scale=0.2] at (axis cs: {cos(-90)},{2*sin(-45)}) (a12) {$(x_{0},y_{0})$};

\draw[blue!50!black,<->]  (a) -- (e);
\draw[blue!50!black,<->,name path=linea]  (b) -- (f); 
\draw[purple!50!black,->,>=stealth,dashed,name path=linea]  (a1) -- (a2); 
\draw[purple!50!black,->,>=stealth,dashed,name path=linea]  (a1) -- (a3); 
\draw[purple!50!black,->,>=stealth,dashed,name path=linea]  (a4) -- (a5); 
\draw[purple!50!black,->,>=stealth,dashed,name path=linea]  (a6) -- (a7); 
\draw[purple!50!black,->,>=stealth,dashed,name path=linea]  (a9) -- (a8); 
\draw[purple!50!black,->,>=stealth,dashed,name path=linea]  (a10) -- (a11); 
\draw[purple!50!black,->,>=stealth,dashed,name path=linea]  (a10) -- (a12); 

\end{axis}

\end{tikzpicture}
  \caption{Projecting onto the cross $C$ when $X=\RR$ (see \cref{ex:cross}).}
  \label{fig}
\end{figure}
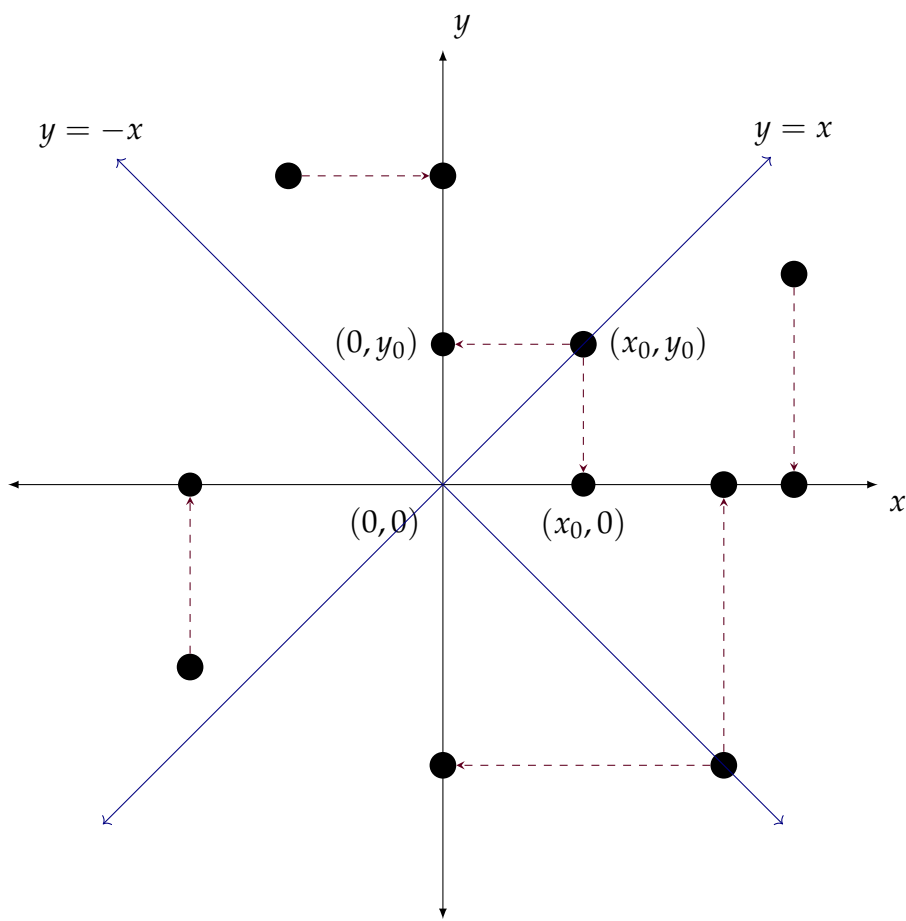

\section*{Acknowledgments}
The authors are grateful to the reviewers and the handling editor for their thoughtful 
and perceptive comments. 
HHB and XW were supported by the Discovery grants of Natural Sciences and
Engineering Research Council of Canada. 
MKL is partially supported by SERB-UBC fellowship and NSERC Discovery grants of HHB and XW.


\begin{thebibliography}{999}
\sepp


\bibitem{BC2017}
H.H.\ Bauschke and P.L.\ Combettes,
\emph{Convex Analysis and Monotone Operator Theory in Hilbert Spaces},
second edition,
Springer, 2017.

\bibitem{Bertsekas}
D.P.\ Bertsekas,
\emph{Nonlinear Programming}, third edition, Athena Scientific, Belmont, Massachusetts, USA, 2016.

\bibitem{Blumenson}
L.E.\ Blumenson, 
A derivation of $n$-dimensional spherical coordinates, 
\emph{The American Mathematical Monthly} 67(1) (1960), 63--66.
\url{https://www.jstor.org/stable/2308932}

\bibitem{BMB} 
E.\ Busseti, W.M.\ Moursi, and S.\ Boyd, 
Solution refinement at regular points of conic problems, 
\emph{Computational Optimization and Applications} 74(3) (2019), 627--643.  
\url{https://doi.org/10.1007/s10589-019-00122-9}



\bibitem{Elser}
V.\ Elser, 
Learning without loss, 
\emph{Fixed Point Theory and Algorithms for Sciences and Engineering} 2021, 
article~12 (2021).
\url{https://doi.org/10.1186/s13663-021-00697-1}



\bibitem{KLT}
A.Y.\ Kruger, D.R.\ Luke, and N.H.\ Thao,
Set regularities and feasibility problems,
\emph{Mathematical Programming (Series B)}~168 (2018), 279--311.
\url{https://doi.org/10.1007/s10107-016-1039-x}


\bibitem{LY}
Y.-C.\ Liang and J.J.\ Ye,
Optimality conditions and exact penalty for mathematical programs 
with switching constraints,
\emph{Journal of Optimization Theory and Applications}~190 (2021), 1--31.
\url{https://doi.org/10.1007/s10957-021-01879-y}

\bibitem{Luenberger}
D.G.\ Luenberger,
\emph{Optimization by Vector Space Methods},
Wiley, 1969.

\end{thebibliography}
\end{document}